\documentclass[12pt,letterpaper]{article}
\usepackage{amsfonts,amssymb,latexsym,amsthm}

\newtheorem{theorem}{Theorem}[section]
\newtheorem{lemma}[theorem]{Lemma}
\newtheorem{definition}[theorem]{Definition}

\newtheorem{remark}[theorem]{Remark}

\newcommand\res{ \mathord{\upharpoonright}}  

\begin{document}

\author{Ramiro de la Vega \footnote{Universidad de los Andes, Bogot\'a,
Colombia,
\ \ rade@uniandes.edu.co}}
\title{Coloring Grids
 \footnote{
   2000 Mathematics Subject Classification: Primary 03E50, Secundary
   03E05, 51M05. Key Words and Phrases: Continuum hypothesis,
   Sierpinski's theorem, n-grids. }
} \maketitle

\begin{abstract}
A structure $\mathcal{A}=\left(A;E_i\right)_{i\in n}$ where each
$E_i$ is an equivalence relation on $A$ is called an
\emph{$n$-grid} if any two equivalence classes coming from
distinct ´$E_i$'s intersect in a finite set. A function $\chi: A
\to n$ is an \emph{acceptable coloring} if for all $i \in n$, the
set $\chi^{-1}(i)$ intersects each $E_i$-equivalence class in a
finite set. If $B$ is a set, then the $n$-cube $B^n$ may be seen
as an $n$-grid, where the equivalence classes of $E_i$ are the
lines parallel to the $i$-th coordinate axis. We use elementary
submodels of the universe to characterize those $n$-grids which
admit an acceptable coloring. As an application we show that if an
$n$-grid $\mathcal{A}$ does not admit an acceptable coloring, then
every finite $n$-cube is embeddable in $\mathcal{A}$.
\end{abstract}

\section{Introduction}

Following \cite{sch2}, for a natural number $n \geq 2$ we shall
call an $n$-\emph{grid} a structure of the form
$\mathcal{A}=\left(A;E_i\right)_{i\in n}$ such that each $E_i$ is
an equivalence relation on the set $A$ and $[a]_i \cap [a]_j$ is
finite whenever $a \in A$ and $i<j<n$ (where $[a]_i$ denotes the
equivalence class of $a$ with respect to the relation $E_i$). An
$n$-\emph{cube} is a particular kind of $n$-grid where $A$ is of
the form $A=A_0 \times \cdots \times A_{n-1}$ and each $E_i$ is
the equivalence relation on $A$ whose equivalence classes are the
lines parallel to the $i$-th coordinate axis (i.e. two $n$-tuples
are $E_i$-related if and only if all of their coordinates coincide
except perhaps for the $i$-th one). An \emph{acceptable coloring}
for an $n$-grid $\mathcal{A}$ is a function $\chi:A \to n$ such
that $[a]_i \cap \chi^{-1}(i)$ is finite for all $a \in A$ and $i
\in n$.

In \cite{sch2}, J.H. Schmerl gives a really nice characterization
of those semialgebraic $n$-grids which admit an acceptable
coloring:

\begin{theorem} \emph{(Schmerl)}
Suppose that $2 \leq n < \omega$, $\mathcal{A}$ is a semialgebraic
$n$-grid and $2^{\aleph_0} \geq \aleph_{n-1}$. Then the following
are equivalent: \begin{itemize} \item[$(1)$] some finite $n$-cube
is not embeddable in $\mathcal{A}$. \item[$(2)$] $\mathbb{R}^n$ is
not embeddable in $\mathcal{A}$. \item[$(3)$] $\mathcal{A}$ has an
acceptable $n$-coloring.
\end{itemize}
\end{theorem}

In this note, we present a characterization that works for any
$n$-grid (see Definition \ref{twisted} and Theorem \ref{main}).
Then we use this characterization to show that $(1) \Rightarrow
(3)$ in the previous theorem holds for arbitrary $n$-grids (see
Theorem \ref{finite cubes}). In fact, the size of the continuum
turns out to be irrelevant for this implication. The implication
$(3) \Rightarrow (2)$ for arbitrary $n$-grids follows from a
result of Kuratowski as it is mentioned in \cite{sch2}. None of
these implications can be reversed for arbitrary $n$-grids,
regardless of the size of the continuum.

\section{Twisted $n$-grids}

In this section we use elementary submodels of the universe to
obtain a characterization of those $n$-grids which admit an
acceptable coloring. At first sight this characterization seems
rather cumbersome, but it is the key to our results in the next
section. The case $n=3$ was already obtained in \cite{ram} with a
bit different terminology and latter used in \cite{sch1}.

As it has become customary, whenever we say that $M$ is an {\it
elementary submodel of the universe}, we really mean that
$(M,\in)$ is an elementary submodel of $(H(\theta),\in)$ where
$H(\theta)$ is the set of all sets of hereditary cardinality less
than $\theta$ and $\theta$ is a large enough regular cardinal
(e.g. when we are studying a fixed $n$-grid $\mathcal{A}$ on a
transitive set $A$, $\theta=(2^{|A|})^+$ is large enough).

Given an equivalence relation $E$ on a set $A$, we say that $B
\subseteq A$ is $E$-\emph{small} if the $E$-equivalence classes
restricted to $B$ are all finite. Note that the $E$-small sets
form an ideal in the power set of $A$. Using this terminology, an
$n$-coloring $\chi:A \to n$ is acceptable for the $n$-grid
$\left(A;E_i\right)_{i\in n}$ if and only if $\chi^{-1}(i)$ is
$E_i$-small for each $i \in n$.

A \emph{test set for} an $n$-grid $\mathcal{A}$ is a set
$\mathcal{M}$ of elementary submodels of the universe such that
$\mathcal{A} \in \bigcap \mathcal{M}$, $|\mathcal{M}|=n-1$ and
$\mathcal{M}$ is linearly ordered by $\in$.

\begin{definition}\label{twisted} We say that that an $n$-grid $\mathcal{A}=\left(A; E_i \right)_{i\in n}$ is {\rm twisted} if for every test set $\mathcal{M}$
for $\mathcal{A}$ and every $k \in n$, the set
$$ \left\{ x \in A \setminus \cup \mathcal{M} : [x]_i \in \cup \mathcal{M} \mbox{ for all } i \neq k
\right\}$$ is $E_k$-small.
\end{definition}

The rest of this section is devoted to show that twisted $n$-grids
are exactly the ones that admit acceptable colorings. For this,
let us fix an arbitrary $n$-grid $\mathcal{A}=\left(A; E_i
\right)_{i\in n}$; our first task is to cover $A$ with countable
elementary submodels in a way that allows us to define a suitable
rank function for elements of $A$ and for $E_i$-equivalence
classes of elements of $A$.

We fix $M_\Lambda$ an elementary submodel such that $A \cup
\{\mathcal{A}\} \subseteq M_\Lambda$ and we let
$\kappa=|M_\Lambda|$. Thinking of $\kappa$ as an initial ordinal,
we let $T=\bigcup_{m \in \omega} \kappa^m$ be the set of finite
sequences of ordinals in $\kappa$. We have two natural orders on
$T$, the tree (partial) order $\subseteq$ and the lexicographic
order $\leq$. In both orders we have the same minimum element
$\Lambda$, the empty sequence. For $\sigma \in T$ and $\alpha \in
\kappa$ we write $\sigma ^\smallfrown \alpha = \sigma \cup \left\{
\left< |\sigma|,\alpha \right> \right\}$. Given $\sigma \in T
\setminus \{\Lambda\}$ we write $\sigma + 1$ for the successor of
$\sigma$ in the lexicographic order of $\kappa^{|\sigma|}$; that
is
$$\sigma+1=\left( \sigma \res (|\sigma|-1) \right) ^\smallfrown
\left( \sigma(|\sigma|-1)+1 \right).$$ We shall write $\sigma
\land \tau$ for the infimum of $\sigma$ and $\tau$ with respect to
the tree order; thus for $\sigma \neq \tau$ we have: $$\sigma
\land \tau = \sigma \res |\sigma \land \tau|=\tau \res |\sigma
\land \tau| \,\,\, \mathrm{ and}$$
$$\sigma(|\sigma \land \tau|) \neq \tau(|\sigma \land
\tau|).$$

Now we can find inductively (on the length of $\sigma \in T$)
elementary submodels $M_\sigma$ such that:

\begin{itemize}

\item[$i$)] The sequence $\left<M_{\sigma ^\smallfrown \alpha} :
\alpha \in cof(|M_\sigma|) \right>$ is a continuous (increasing)
elementary chain,

\item[$ii$)] $M_\sigma \subseteq \bigcup \left\{ M_{\sigma
^\smallfrown \alpha} : \alpha \in cof(|M_\sigma|) \right \}$,

\item[$iii$)] $\{\mathcal{A}\} \cup \left\{ M_\tau : \tau+1
\subseteq \sigma \right \} \subseteq M_{\sigma ^\smallfrown 0}$,
and

\item[$iv$)] If $\tau \subsetneq \sigma$ and $M_\tau$ is
uncountable then $|M_\tau|>|M_\sigma|$.

\end{itemize}

We actually do not need to (and will not) define $M_{\sigma
^\smallfrown \alpha}$ when $M_\sigma$ is countable or if $\alpha
\geq cof(|M_\sigma|)$.

Although the lexicographic order on $T$ is not a well order, it is
not hard to see that conditions $ii$ and $iv$ allow the following
definition of rank to make sense:

\begin{definition}For $x \in M_\Lambda$ we define $rk(x)$ as the minimum
$\sigma \in T$ (in the lexicographic order) such that $M_\sigma$
is countable and $x \in M_\tau$ for all $\tau \subseteq \sigma$.
\end{definition}

Note that by the continuity of the elementary chains in condition
$i$, we have that $rk(x)$ is always a finite sequence of ordinals
which are either successor ordinals or $0$. In particular, if
$\sigma_x=rk(x)$, $\sigma_y=rk(y)$, $\sigma_x < \sigma_y$ and
$m=|\sigma_x \land \sigma_y|$, then $\sigma_y(m)$ is a successor
ordinal say $\alpha+1$ and we can define
$$\Delta(x,y)=(\sigma_x \land \sigma_y)^\smallfrown \alpha.$$

This last definition will only be used in the proof of Lemma
\ref{almenos}. The following remark summarizes the basic
properties of $\Delta(x,y)$ that we will be using; all of them
follow rather easily from the definitions.

\begin{remark}\label{delta} If $rk(x)<rk(y)$ then
\begin{itemize}
\item $x \in M_{\Delta(x,y)}$ and $y \notin M_{\Delta(x,y)}$,
\item $\Delta(x,y)+1 \subseteq rk(y)$, \item if $\sigma \supsetneq
\Delta(x,y)+1$ then $M_{\Delta(x,y)} \in M_\sigma$ (by conditions
$i$ and $iii$).
\end{itemize}
\end{remark}

After assigning a rank to each member of $M_\Lambda$, we need a
way to order in type $\omega$ all the elements of $M_\Lambda$ of
the same rank. This is easily done by fixing an injective
enumeration
$$M_\sigma=\left\{ t^\sigma_m : m \in \omega \right\}$$ for each
$\sigma$ for which $M_\sigma$ is countable, and defining the
degree of an element of $M_\Lambda$ as follows:

\begin{definition}
For $x \in M_\Lambda$ we define $deg(x)$ as the unique natural
number satisfying $$x = t^{rk(x)}_{deg(x)}.$$
\end{definition}

The following two lemmas will be used to construct an acceptable
coloring for $\mathcal{A}$ in the case that $\mathcal{A}$ is
twisted, although the second one does not make any assumptions on
$\mathcal{A}$.

\begin{lemma}\label{almenos} If $\mathcal{A}$ is twisted then there is a set $B \subseteq A$ and a partition $B= \bigcup_{k \in n} B_k$ such that:
\begin{itemize}
\item[$a)$] Each $B_k$ is $E_k$-small and \item[$b)$] $|\{i \in n:
rk([x]_i)=rk(x)\}| \geq 2$ for any $x \in A \setminus B$.
\end{itemize}
\end{lemma}

\begin{proof}
For each $k \in n$ we let $B_k$ be the set of all $x \in  A$ such
that $rk([x]_k)>rk([x]_i)$ for all $i \neq k$. Let $B=\bigcup_{k
\in n} B_k$.

Note that for any $x \in A$ and $i \in n$ we have that $rk([x]_i)
\leq rk(x)$. On the other hand if $\sigma=rk([x]_k)=rk([x]_j)$ for
some $k \neq j$, then by elementarity and the fact that $[x]_k
\cap [x]_j$ is finite, it follows that $rk(x)\leq \sigma$ and
hence $rk(x)=\sigma$. This observation easily implies that
condition $b)$ is met. It also implies that if $x \in B_k$ then
$$rk([x]_{k_0})< \cdots < rk([x]_{k_{n-2}}) < rk([x]_k) \leq rk(x)$$
for some numbers $k_0,\dots,k_{n-2}$ such that
$\{k_0,\dots,k_{n-2},k\}=n$.

Now we put $\mathcal{M}=\left\{ M_{\Delta([x]_{k_i},x)} : i \in
n-1 \right\}$, and use $\mathcal{M}$ as a test set for
$\mathcal{A}$ to conclude that, since $\mathcal{A}$ is twisted,
$B_k$ is $E_k$-small.

To see that $\mathcal{M}$ is indeed a test set, it is enough to
show that $M_{\Delta([x]_{k_i},x)} \in M_{\Delta([x]_{k_j},x)}$
for $i<j$. So fix $i<j$ and note that since $[x]_{k_i} \cap
[x]_{k_j}$ is finite we have
$\Delta([x]_{k_i},x)=\Delta([x]_{k_i},[x]_{k_j})$ and therefore by
Remark \ref{delta}, $$\Delta([x]_{k_i},x)+1 \subseteq rk(x)\land
rk([x]_{k_j}).$$ But then $\Delta([x]_{k_i},x)+1 \subsetneq
\Delta([x]_{k_j},x)$ and again by Remark \ref{delta} we get
$M_{\Delta([x]_{k_i},x)} \in M_{\Delta([x]_{k_j},x)}$.

\end{proof}

\begin{lemma}\label{finitud}For all $i,k \in n$ with $i \neq k$, the
set
$$C_{i,k}=\left\{ x \in A : rk([x]_i) = rk([x]_k) \,\,\mathrm{and}\,\, deg([x]_i) < deg([x]_k) \right\}$$ is $E_k$-small.
\end{lemma}

\begin{proof}
Fix $a \in A$ and let $\sigma=rk([a]_k)$ and $d=deg([a]_k)$. Note
that if $x \in C_{i,k} \cap [a]_k$ then there is an $m<d$ (namely
$m=deg([x]_i)$) such that $x \in t^\sigma_m \cap t^\sigma_d$ and
$t^\sigma_m \cap t^\sigma_d$ is finite. Hence $C_{i,k} \cap [a]_k$
is contained in a finite union of finite sets.

\end{proof}

We are finally ready to prove the main result of this section.

\begin{theorem}\label{main}  The following are equivalent:

\begin{itemize}
\item[$1)$] $\mathcal{A}$ is twisted.

\item[$2)$] $\mathcal{A}$  admits an acceptable coloring.
\end{itemize}
\end{theorem}

\begin{proof}
Suppose first that $\mathcal{A}$ is twisted. Let $B$ and $B_k$ for
$k \in n$ be as in Lemma \ref{almenos}, and let $C_{i,k}$ for $i,k
\in n$ be as in Lemma \ref{finitud}. For each $k \in n$ define
$C_k$ as the set of all $x \in A \setminus B$ such that:

\begin{itemize}

\item[$i$)] $rk(x)=rk([x]_k)$, and

\item[$ii$)] for all $i \in n \setminus \{k\}$, if
$rk([x]_i)=rk([x]_k)$ then $deg([x]_i) < deg([x]_k)$.

\end{itemize}

By condition $b)$ in Lemma \ref{almenos}, we have that $C_k
\subseteq \bigcup_{i \in n}C_{i,k}$ and therefore each $C_k$ is
$E_k$-small. It also follows that the $C_k$'s form a partition of
 $A \setminus B$ so that we can define an acceptable coloring for $\mathcal{A}$ by:
$$\chi(x)=k \mbox{ if and only if } x \in B_k \cup C_k.$$

Now suppose that $\mathcal{A}$ admits an acceptable coloring and
fix a test set $\mathcal{M}$ and $k \in n$. We want to show that
the set
$$ X = \left\{ x \in A \setminus \cup \mathcal{M} : [x]_i \in \cup \mathcal{M} \mbox{ for all } i \neq k
\right\}$$ is $E_k$-small. For this let $\chi:A \to n$ be an
acceptable coloring such that (using elementarity and the fact
that $\mathcal{M}$ is linearly ordered by $\in$) $\chi$ belongs to
each $M \in \mathcal{M}$. Now if $x \in X$ and $i \neq k$ then
there is an $M \in \mathcal{M}$ such that $[x]_i \cap \chi^{-1}(i)
\in M$ and hence $[x]_i \cap \chi^{-1}(i) \subset M$ (since $\chi$
is acceptable); this implies that $\chi(x) \neq i$. It follows
that $X \subseteq \chi^{-1}(k)$ so that $X$ is $E_k$-small.

\end{proof}

\section{Embedding cubes into $n$-grids}

Given an $n$-grid $\mathcal{A}=\left(A; E_i \right)_{i\in n}$ it
will be convenient in this section to have a name $\rho_i: A \to A
/ E_i$ for the quotient maps ($\rho_i(\cdot)=[\cdot]_i$). Note
that if $i \neq k$, $C \subseteq A$ is infinite and $\rho_k
\upharpoonright C$ is constant, then there is an infinite $D
\subseteq C$ such that $\rho_i \upharpoonright D$ is injective. We
will make repeated use of this fact without explicitly saying so,
in the proof of the following:

\begin{theorem}\label{finite cubes}
If $\mathcal{A}$ is a non-twisted $n$-grid then any finite
$n$-cube $l^n$ (with $l \in \omega$) can be embedded in
$\mathcal{A}$.
\end{theorem}

\begin{proof}
By definition, since $\mathcal{A}$ is not twisted, there is a test
set $\mathcal{M}$ and a $k \in n$ such that for some $a \in A$,
the set
$$ B=\left\{ x \in [a]_k \setminus \cup \mathcal{M} : [x]_i \in \cup
\mathcal{M} \mbox{ for all } i \neq k \right\}$$ is infinite. For
each $x \in B$ and each $i \in n \setminus \{k\}$ there is an
$M^x_i \in \mathcal{M}$ such that $[x]_i \in M^x_i$. Since
$\mathcal{M}$ is finite, there must be an infinite $C \subseteq B$
on which the map $x \mapsto \left<M^x_i : i \in n \setminus \{k \}
\right>$ is constant, say with value $\left<M_i : i \in n
\setminus \{k \} \right>$. Note that since $C$ is disjoint from
$\cup \mathcal{M}$, the map $i \mapsto M_i$ must be injective and
hence $\mathcal{M}=\{M_i : i \in n \setminus \{k\} \}$, because
$|\mathcal{M}|=n-1$. Finally, we can find an infinite set $D
\subseteq C$ such that $\rho_i \upharpoonright C$ is injective for
all $i \neq k$.

Now taking $k_1=k$ and letting $\varphi$ be any injection from $l$
into $D$, we easily see that the following statement is true for $j=1$: \\

\emph{$\textbf{P(j)}$: There are distinct $k_1,\dots,k_j \in n$
and an embedding $\varphi:l^j \to (A; E_{k_1},\dots,E_{k_j})$ such
that:
\begin{itemize}
\item[$a)$] for $i\in n \setminus \{k_1,\dots,k_j \}$, $\rho_i
\circ \varphi$ is injective and belongs to $M_i$, \item[$b)$]
$\varphi$ takes values in $A \setminus \bigcup \left\{M_i : i \in
n \setminus \{k_1,\dots,k_j \} \right\}$.
\end{itemize}
}

Note that when $j=n$, conditions $a)$ and $b)$ become trivially
true, and $P(n)$ just says that there is an embedding (modulo an
irrelevant permutation of coordinates) of the finite cube $l^n$
into $\mathcal{A}$, which is exactly what we want to show. We
already know that $P(1)$ is true, so we are done if we can show
that $P(j)$ implies $P(j+1)$ for $1 \leq j<n$.

Assuming $P(j)$, let $\varphi:l^j \to (A; E_{k_1},\dots,E_{k_j})$
be such an embedding, and let $k_{j+1} \in n \setminus
\{k_1,\dots,k_j \}$ be such that $M_{k_{j+1}}$ is the
$\in$-maximum element of $\left\{M_i : i \in n \setminus
\{k_1,\dots,k_j \} \right\}$. Let us call $$\delta:=\rho_{k_{j+1}}
\circ \varphi \in M_{k_{j+1}}.$$

Now note that $\varphi \notin M_{k_{j+1}}$ and at the same time
$\varphi$ satisfies the following properties (on the free variable
$\Phi$), all of which can be expressed using parameters from
$M_{k_{j+1}}$:
\begin{itemize}
\item $\Phi:l^j \to (A; E_{k_1},\dots,E_{k_j})$ is an embedding,
\item $\rho_{k_{j+1}} \circ \Phi = \delta$, \item for $i\in n
\setminus \{k_1,\dots,k_j, k_{j+1} \}$, $\rho_i \circ \Phi$ is
injective and belongs to $M_i$, \item $\Phi$ takes values in $A
\setminus \bigcup \left\{M_i : i \in n \setminus \{k_1,\dots, k_j,
k_{j+1} \} \right\}$.
\end{itemize}

This means that there must be an infinite set (in fact there must
be an uncountable one, but we won't be using this) $\{\varphi_m :
m \in \omega\}$ of distinct functions satisfying those properties.
Going to a subsequence $l^j$-many times, we may assume without
loss of generality that for each $t \in l^j$, the map $m \mapsto
\varphi_m(t)$ is either constant or injective. Now since they
cannot all be constant, it is not hard to see that in fact all
these maps have to be injective: just note that if $t,t' \in l^j$
are in a line parallel to the $(r-1)$-th coordinate axis then it
cannot be the case that the map associated with $t$ is constant
while the one associated with $t'$ is injective, since otherwise
$\{\varphi_m(t'):m \in \omega \}$ would be an infinite set
contained in $[\varphi_0(t)]_{k_r} \cap
[\varphi_0(t')]_{k_{j+1}}$. To see this, just note that in that
situation we would have
$[\varphi_m(t')]_{k_r}=[\varphi_m(t)]_{k_r}=[\varphi_0(t)]_{k_r}$
and $[\varphi_m(t')]_{k_{j+1}}=(\rho_{k_{j+1}} \circ
\varphi_m)(t')= \delta(t') = (\rho_{k_{j+1}} \circ
\varphi_0)(t')=[\varphi_0(t')]_{k_{j+1}}$.

Next we can find an infinite $I \subseteq \omega$ such that for
each $t \in l^j$ and each $i \in n \setminus \{k_{j+1}\}$ the map
$m \mapsto [\varphi_m(t)]_i$ is injective when restricted to $I$.
From here one can find (one at a time) $l$ distinct elements
$m_0,\dots,m_{l-1}$ of $I$ such that for all $t,t' \in l^j$, for
all $r,r' \in l$ with $r \neq r'$ and for all $i \in n \setminus
\{k_{j+1}\}$, we have that $[\varphi_{m_r}(t)]_i \neq
[\varphi_{m_{r'}}(t')]_i$.

Finally we let $\psi :l^{j+1} \to (A; E_{k_1},\dots,E_{k_{j+1}})$
be the function defined by $\psi(t,r)=\varphi_{m_r}(t)$. By the
way that we constructed the $m_r$'s and using the fact that all
the $\varphi_m$'s are embeddings and also using that $\delta$ is
injective, one can see that $\psi$ is in fact an embedding. From
the fact that $\psi$ is essentially a finite union of some
$\varphi_m$'s and by the way we chose those $\varphi_m$'s, it
follows that conditions $a)$ and $b)$ in $P(j+1)$ are satisfied.

\end{proof}

This last theorem only goes one way: for example, the $n$-cube
$\omega^n$ is twisted for $n \geq 2$, but of course any finite
$n$-cube can be embedded in it. I suspect that only for very
``nice" classes of $n$-grids one can reverse this theorem.
Schmerl's theorem does it for semialgebraic $n$-grids; perhaps
some form of o-minimality is what is required.\\

The question of when can an infinite cube be embedded in an
arbitrary $n$-grid seems more subtle. For instance, let us
consider the case $n=2$. Using the same idea as for the proof of
\ref{finite cubes}, one can easily show:

\begin{theorem} If $\mathcal{A}$ is a non-twisted $2$-grid then either $l \times \omega_1$ can be
embedded in $\mathcal{A}$ for all $l \in \omega$, or $\omega_1
\times l$ can be embedded in $\mathcal{A}$ for all $l \in \omega$.
\end{theorem}

However, it is not true that $\omega \times \omega$ embeds in any
non-twisted $2$-grid. For example, fix an uncountable family
$\left\{ A_\alpha : \alpha \in \omega_1 \right\}$ of almost
disjoint subsets of $\omega$ and let $A=\left\{(n,\alpha) \in
\omega \times \omega_1 : n \in A_\alpha \right\}$. Think of $A$ as
a subgrid of the $2$-cube $\omega \times \omega_1$. It is easy to
see that this is a non-twisted grid, but not even $\omega \times
2$ can be embedded in it.

\end{document}